\documentclass[11pt]{article}
\usepackage{amsmath}
\usepackage{amsthm}
\usepackage{amssymb,amsfonts}
\usepackage{hyperref}
\usepackage{latexsym}
\usepackage{todonotes}
\usepackage{nicefrac}
\usepackage{epsfig}
\usepackage{stmaryrd}
\usepackage{setspace}
\usepackage{enumerate}
\usepackage[all]{xypic}
\usepackage{bbm,ifpdf,tikz}
\ifpdf
\usepackage{pdfsync}
\fi

\newtheorem{theorem}{Theorem}[section]

\newtheorem{lemma}[theorem]{Lemma}
\newtheorem{proposition}[theorem]{Proposition}
\newtheorem{corollary}[theorem]{Corollary}

\newtheorem{conjecture}[theorem]{Conjecture}

{
\theoremstyle{definition}

\newtheorem{remark}[theorem]{Remark}
}

\newcommand{\excise}[1]{}

\renewcommand{\dim}{\operatorname{dim}}

\newcommand{\Sym}{\operatorname{Sym}}
\newcommand{\rk}{\operatorname{rk}}
\newcommand{\crk}{\operatorname{crk}}

\renewcommand{\and}{\qquad\text{and}\qquad}
\newcommand{\Ind}{\operatorname{Ind}}

\newcommand{\Hom}{\operatorname{Hom}}

\newcommand{\ch}{\operatorname{ch}}

\newcommand{\Z}{\mathbb{Z}}
\newcommand{\Q}{\mathbb{Q}}
\newcommand{\N}{\mathbb{N}}

\newcommand{\C}{\mathbb{C}}

\renewcommand{\cH}{\mathcal{H}}
\newcommand{\cK}{\mathcal{K}}
\newcommand{\cQ}{\mathcal{Q}}
\renewcommand{\cR}{\mathcal{R}}
\newcommand{\cP}{\mathcal{P}}
\newcommand{\cC}{\mathcal{C}}
\newcommand{\cI}{\mathcal{I}}

\newcommand{\la}{\lambda}

\newcommand{\OS}{OS}
\newcommand{\grRep}{\operatorname{grRep}}
\newcommand{\grVRep}{\operatorname{grVRep}}
\newcommand{\Rep}{\operatorname{Rep}}
\newcommand{\VRep}{\operatorname{VRep}}
\newcommand{\OSc}{H}

\newcommand{\nicktodo}{\todo[inline,color=green!20]}

\begin{document}
\spacing{1.2}
\noindent{\Large\bf The equivariant Kazhdan-Lusztig polynomial of a matroid}\\

\noindent{\bf Katie Gedeon, Nicholas Proudfoot,
and Benjamin Young}\\
Department of Mathematics, University of Oregon,
Eugene, OR 97403\\



{\small
\begin{quote}
\noindent {\em Abstract.}
We define the equivariant Kazhdan-Lusztig polynomial of a matroid equipped with a group of symmetries,
generalizing the nonequivariant case.  We compute this invariant for arbitrary uniform matroids and for braid
matroids of small rank.
\end{quote} }

\section{Introduction}
The Kazhdan-Lusztig polynomial $P_M(t)\in\Z[t]$ of a matroid $M$ was introduced in \cite{EPW}.
In the case where $M$ is realizable by a linear space $V\subset\C^n$, the coefficient of $t^i$ in $P_M(t)$
is equal to the dimension of the intersection cohomology group $I\! H^{2i}(X_V; \C)$, where $X_V$
is the ``reciprocal plane" of $V$ \cite[Proposition 3.12]{EPW}.  In particular, this implies that $P_M(t)\in \N[t]$
whenever $M$ is realizable.  We conjectured \cite[Conjecture 2.3]{EPW} that $P_M(t)\in \N[t]$ for every matroid $M$.
We also gave some computations of $P_M(t)$ for uniform matroids and braid matroids of small rank.

The purpose of this paper is to define a more refined invariant.
Given a matroid $M$ equipped with an action of a finite group $W$, we define the equivariant Kazhdan-Lusztig polynomial
$P_M^W(t)$.  The coefficients of this polynomial are not integers, but rather virtual representations of the group $W$.
If $W$ is the trivial group, the ring of virtual representations of $W$ is $\Z$, and $P_M^W(t)$ is equal to the ordinary polynomial $P_M(t)$.
More generally, the polynomial $P_M(t)$ may be obtained from $P_M^W(t)$ by sending a virtual representation to its dimension.
If $M$ is equivariantly realizable by a linear space $V\subset\C^n$, the coefficient of $t^i$ in $P_M^W(t)$
is equal to the intersection cohomology group $I\! H^{2i}(X_V; \C)$, regarded as a representation of $W$ (Corollary \ref{rep}).
In particular, this implies that the coefficients of $P_M^W(t)$ are honest (rather than virtual) representations of $W$ whenever
$M$ is equivariantly realizable.  We conjecture that this is the case even in the non-realizable case (Conjecture \ref{positivity}).
We compute the coefficients of $P_M^W(t)$ for arbitrary uniform matroids (Theorem \ref{uniform}) and for braid matroids
of small rank (Section \ref{sec:calculations}).\\

It is reasonable to ask why bother with an equivariant version of this invariant, especially since there are
still many things that we do not understand about the nonequivariant version.  
We have four answers to this question, all of which are illustrated by the case of uniform matroids.
To set notation,
let $U_{m,d}$ be the uniform matroid of rank $d$ on a set of $m+d$ elements, which is equipped
with a natural action of the symmetric group $S_{m+d}$.  Let $C_{i,m,d}$ be the coefficient of $t^i$ in the equivariant
Kazhdan-Lusztig polynomial of $U_{m,d}$, and  let $c_{i,m,d} = \dim C_{i,m,d}$ be the coefficient
of $t^i$ in the nonequivariant Kazhdan-Lusztig polynomial.

\begin{itemize}
\item {\bf Nicer formulas:}  Our formula for $C_{i,m,d}$ (Theorem \ref{uniform}) is very simple;
it is a multiplicity-free sum of irreducible representations that are easy to describe.  We could of course
use the hook-length formula for the dimension of an irreducible representation of $S_{m+d}$
to derive a formula for $c_{i,m,d}$, but the resulting formula is messy and unenlightening.
Indeed, we computed a table in the appendix of \cite{EPW} consisting of the numbers $c_{i,m,d}$
for small values of $i$, $m$, and $d$, and at that time we were unable even to guess the general formula.  
It was only by keeping track of the extra structure that we were able to see the essential pattern.
\item {\bf More powerful tools:}  After we figured out the correct statement of Theorem \ref{uniform},
we attempted to prove the formula for $c_{i,m,d}$ directly (without going through Theorem \ref{uniform}), and we
failed.  The Schubert calculus techniques that we employ in the proof of Theorem \ref{uniform}
are considerably more powerful than the tools to which we have access in the nonequivariant setting.
\item {\bf Representation stability:}  The sequence of representations $C_{i,m,d}$ is uniformly
representation stable in the sense of Church and Farb \cite{CF}, which essentially means that
it admits a description that is independent of $d$, provided that $d\geq m+2i$ (Remark \ref{stable}).
This phenomenon cannot be seen by looking at the numbers $c_{i,m,d}$.
\item {\bf Non-realizable examples:}  It is difficult to write down examples of
non-realizable irreducible matroids for which we can compute the Kazhdan-Lusztig polynomial, and therefore
we had no nontrivial checks of our non-negativity conjecture in the non-equivariant setting.
On the other hand, the uniform matroid $U_{m,d}$ is equivariantly non-realizable provided that both $d$ and $m$
are greater than 1.  This means that Theorem \ref{uniform} provides good evidence for Conjecture \ref{positivity},
and therefore by extension for \cite[Conjecture 2.3]{EPW}.
\end{itemize}

The paper is structured as follows.  In Section \ref{sec:def}, we define the equivariant characteristic polynomial
and use it to define the equivariant Kazhdan-Lusztig polynomial.  This section closely mirrors Section 2 of \cite{EPW},
but some of the basic lemmas are much more technical in the equivariant setting.  In particular, Lemma \ref{dual}
is an equivariant version of a well-known statement that is usually proved via M\"obius inversion.  This proof
does not work in the equivariant context (due essentially to the fact that the equivariant analogue of the M\"obius
algebra is not associative), so we needed to find a different approach.

Section \ref{sec:uniform} is devoted to the study of uniform matroids, and in particular the statement
and proof of Theorem \ref{uniform}.  Our main technique is to express everything in terms of generating
functions that encode all three parameters $i$, $m$, and $d$, and then to manipulate our functional equations
until they can be solved using repeated applications of the Pieri rule.  Section \ref{sec:braid} treats the case
of braid matroids.  In this case we are not able to give a general formula for the equivariant Kazhdan-Lusztig
polynomial, but we do derive generating function identities that allow us to compute the polynomial
explicitly in small rank.

Finally, in Section \ref{sec:elc} we introduce the notion of equivariant log concavity, which is a generalization
of the usual notion of log concavity to the equivariant setting.  The statement that the coefficients of the characteristic
polynomial of a matroid form a log concave sequence goes back to the 1960s, and was only recently proved
by Adiprasito, Huh, and Katz \cite{AHK}.  The statement that the coefficients of the Kazhdan-Lusztig polynomial of
a matroid form a log concave sequence was conjectured in \cite[Conjecture 2.5]{EPW}.  Here we make the two
analogous conjectures in the equivariant setting (Conjecture \ref{lc}), and we prove equivariant log concavity
of the characteristic polynomial of a uniform matroid (Proposition \ref{uniform-lc}).  The notion of equivariant
log concavity will be further developed in a future paper.

\vspace{\baselineskip}
\noindent
{\em Acknowledgments:}
The authors are grateful to Max Wakefield for his help in initiating
this project, and to June Huh and David Speyer for helpful conversations.
NP was supported by NSF grants DMS-0950383 and DMS-1565036.  

\section{Definition}\label{sec:def}
Let $M$ be a matroid on the ground set $\cI$, and let $W$ be a finite group acting on $\cI$ and preserving $M$.
We will refer to this collection of data as an {\bf equivariant matroid} $W\curvearrowright M$.
Let $$\grVRep(W) := \VRep(W)\otimes_\Z \Z[t] \and \grRep(W) := \Rep(W)\otimes_\N \N[t].$$
Note that, for any group homomorphism $\varphi:W'\to W$, we obtain ring maps $$\varphi^*:\VRep(W)\to\VRep(W')
\and \varphi^*:\grVRep(W)\to\grVRep(W')$$
taking honest representations to honest representations.

\subsection{The equivariant characteristic polynomial}
Let $\OS^W_{M,i}\in \Rep(W)$ be the degree $i$ part of the Orlik-Solomon algebra of $M$.
We define the {\bf equivariant characteristic polynomial}
$$\OSc^W_M(t) := \sum_{p=0}^{\rk M} (-1)^p t^{\rk M - p} \OS^W_{M,p} \in \grVRep(W).$$
Note that the graded dimension of $H_M^W(t)$ is just the usual characteristic polynomial $\chi_M(t)\in \Z[t]$.
The following lemma is an equivariant version of the statement that $\chi_M(1)=0$ for any matroid $M$ of positive rank.

\begin{lemma}\label{one}
For any equivariant matroid $W\curvearrowright M$ of positive rank, $H_M^W(1) = 0$.
\end{lemma}

\begin{proof}
Let $e = \sum_{i\in\cI} e_i \in \OS_{M,1}$, and consider the complex of $W$-representations
with $i^\text{th}$ term $\OS^W_{M,i}$ and with differential given by multiplication by $e$.
Then $H_M^W(1)$ is equal to the Euler characteristic of this complex, which is equal to the Euler
characteristic of its homology.  But the homology is zero provided that $M$ has positive rank \cite[2.1]{YuzBOS}.
\end{proof}

Let $L$ be the lattice of flats of $M$.
Given a flat $F\in L$, let $W_F\subset W$ be the stabilizer of $F$.
For any pair of flats $F,G\in L$, let $W_{FG} := W_F\cap W_G$.
Let $M_F$ be the {\bf localization} of $M$ at $F$; this is the matroid on the ground set $F$ whose lattice of flats is isomorphic
to $L_F := \{G\in L\mid G\leq F\}$.  Dually, let $M^F$ be the {\bf restriction} of $M$ to $F$; this is the matroid on the ground
set $\mathcal{I}\smallsetminus F$ whose lattice of flats is isomorphic
to $L^F := \{G\in L\mid G\geq F\}$.  The action of $W$ on $M$ induces an action of $W_F$ on both $M_F$ and $M^F$.

\begin{lemma}\label{Brieskorn}
For any equivariant matroid $W\curvearrowright M$,
\begin{eqnarray*}H_M^W(t) &=& \sum_{[F]\in L/W}(-1)^{\rk F} t^{\crk F}\Ind_{W_F}^W\!\left(\OS^{W_F}_{M_F,\,\rk M_F}\right)\\
&=& \sum_{F\in L}\frac{|W_F|}{|W|}(-1)^{\rk F} t^{\crk F}\Ind_{W_F}^W\!\left(\OS^{W_F}_{M_F,\,\rk M_F}\right).\end{eqnarray*}
\end{lemma}

\begin{proof}
Brieskorn's lemma says that the natural map $$\bigoplus_{\substack{F\in L\\ \rk F = p}}\OS_{M_F,p}\to \OS_{M,p}$$
is an isomorphism.  When we incorporate the action of $W$, this map gives us the equation
$$\OS_{M,p}^W \;\;= \sum_{\substack{[F]\in L/W\\ \rk F = p}}\Ind_{W_F}^W\!\left(\OS^{W_F}_{M_F,p}\right)
= \sum_{\substack{F\in L\\ \rk F = p}}\frac{|W_F|}{|W|}\Ind_{W_F}^W\!\left(\OS^{W_F}_{M_F,p}\right).$$
Our statement follows immediately from this.
\end{proof}

\begin{lemma}\label{easy}
For any equivariant matroid $W\curvearrowright M$ of positive rank,
$$\sum_{[F]\in L/W}(-1)^{\rk F} \Ind_{W_F}^W\!\left(\OS^{W_F}_{M_F,\,\rk M_F}\right) 
= \sum_{F\in L}\frac{|W_F|}{|W|}(-1)^{\rk F} t^{\crk F}\Ind_{W_F}^W\!\left(\OS^{W_F}_{M_F,\,\rk M_F}\right) = 0.$$
\end{lemma}

\begin{proof}
This follows immediately from 
Lemmas \ref{one} and \ref{Brieskorn}.
\end{proof}

Lemma \ref{easy} is an equivariant version of the statement that 
\begin{equation*}\label{mob-left}\sum_{F\in L}\mu(\emptyset, F) = 0\end{equation*}
when $M$ has positive rank.  
There is also a dual statement, which says
that \begin{equation*}\label{mob-right}\sum_{F\in L}\mu(F, \cI) = 0\end{equation*}
when $M$ has positive rank.  Lemma \ref{dual} is an equivariant version of this dual equation.  Surprisingly, the proof of Lemma \ref{dual} 
is much more difficult than the proof of Lemma \ref{easy}.

\begin{lemma}\label{dual}
For any equivariant matroid $W\curvearrowright M$ of positive rank, 
$$\sum_{[F]\in L/W} (-1)^{\crk F} \Ind_{W_F}^W\left(\OS^{W_F}_{M^F\!, \,\crk F}\right)
= \sum_{F\in L}\frac{|W_F|}{|W|} (-1)^{\crk F} \Ind_{W_F}^W\left(\OS^{W_F}_{M^F\!, \,\crk F}\right) = 0.$$
\end{lemma}

\begin{proof}
For any flat $F$, we have $$\OS_{M^F\!, \,\crk F} = \C\{e_S\mid \text{$S$ a basis for $M^F$}\}\;\Big{/}\;\C\{\partial e_C\mid \text{$C$ a circuit for $M^F$ of rank $\crk F$}\}.$$
If $F\leq G$ with $\crk F = p$ and $\crk G = p-1$, we define a map $$\varphi^F_G:\OS_{M^F\!,\, p}\to \OS_{M^G\!,\, p-1}$$ 
by the formula $\varphi^F_G(e_S) := \partial e_S$ for any basis $S$ of $M^F$, where we implicitly set $e_i=0$ for all $i\in G$.  
More precisely, we note that $S$ can contain at most one element of $G$.
If $S$ contains no elements of $G$, then $\varphi^F_G(e_S) := \partial e_S = 0\in \OS_{M^G\!,\, p-1}$.  If $S = \{i_1,\ldots,i_r\}$ and $i_k\in G$, then 
we put $S_k:= S\smallsetminus \{i_k\}$ and $\varphi^F_G(e_S) := (-1)^k e_{S_k}$.  This is well defined because $\partial^2 = 0$.

Let $$C_p(M) := \bigoplus_{\crk F = p}\OS_{M^F\!,\, p},$$
and combine the various maps $\varphi^F_G$ to obtain a map $\varphi_p:C_p(M)\to C_{p-1}(M)$.  We claim that $(C_\bullet(M), \varphi_\bullet)$
is an exact sequence.

To show that $\varphi_{p}\circ\varphi_{p+1} = 0$, we need to show that, for all $E\leq G$ with $\crk E = p+1$ and $\crk G = p-1$,
we have $$\sum_{E<F<G}\varphi^F_G\circ\varphi^E_F = 0.$$
Let $S$ be a basis for $M^E$.  Then $\varphi^F_G\circ\varphi^E_F(e_S) = 0$ unless $F$ contains exactly one element
of $S$ and $G$ contains exactly two elements of $S$.  Thus we can reduce to the situation where $S = \{i_1,\ldots,i_r\}$, $F_k$ is the flat spanned
by $S$ and $i_k$, $F_\ell$ is the flat spanned by $S$ and $i_\ell$, and $G$ is the flat spanned by $S$, $i_k$, and $i_\ell$, and we need to show
that $\varphi^{F_k}_G\circ\varphi^E_{F_k}(e_S) + \varphi^{F_\ell}_G\circ\varphi^E_{F_\ell}(e_S) = 0$.  This is easily checked by hand.
Thus $(C_\bullet(M), \varphi_\bullet)$ is a complex.

To prove that our complex is exact, we proceed by induction on $\rk M$.  The case $\rk M = 1$ is trivial.  Fix an $M$ of rank strictly greater than
1, and assume that the statement is proved for all smaller ranks.  Choose an index $i\in\cI$, and consider the sum
$$C'_\bullet(M):= \bigoplus_{i\in F}\OS_{M^F\!,\, \crk F}\subset C_\bullet(M)$$
ranging over all flats $F$ that contain the index $i$.
It is clear that this is a subcomplex, and that $$C'_\bullet(M)\cong C_\bullet(M'),$$ where $M'$ is the restriction of $F$ to the unique flat of rank 1
containing $i$.  Let $$C''_\bullet(M):= C_\bullet(M)/ C'_\bullet(M)$$ be the quotient complex.
As a vector space, we have $$C''_p(M) \cong \bigoplus_{\substack{\crk F = p\\ i\notin F}}\OS_{M^F\!,\, p}.$$
Furthermore, for each flat $F$ of corank $p$ that does not contain $i$, we have an isomorphism
$$\bigoplus_{\substack{i\notin G\geq F\\ \crk G = 1}}\OS_{M_G^F\!,\, p-1}\;\to \;\OS_{M^F\!,\, p}$$
given by multiplication by $e_i$.  (Indeed, if we choose an order on $\cI$ such that $i$ is the maximal element, then
multiplication by $e_i$ gives a bijection from the nbc
basis for the left-hand side to the nbc basis for the right-hand side.\footnote{The abbreviation nbc stands for ``no broken circuit";
see \cite[Theorem 2.8]{YuzOS} for a discussion of this basis.})
These isomorphisms fit together into an isomorphism of complexes $$C''_\bullet(M)\cong \bigoplus_{\substack{i\notin G\\ \crk G = 1}} C_{\bullet-1}(M_G).$$
Now consider the short exact sequence of complexes
$$0\to C'_\bullet(M)\to C_{\bullet}(M)\to C''_{\bullet}(M)\to 0.$$
Since $\rk M' = \rk M_G = \rk M - 1 > 0$, our inductive hypotheses imply that $C'_\bullet(M)$ and $C''(M)$ both have
trivial homology.  Then the long exact sequence in homology tells us that homology of $C_\bullet(M)$ vanishes, as well.

Finally, we note that the complex $C_\bullet(M)$ admits an action of $W$ with $$C_p(M) = \sum_{\substack{[F]\in L/W\\ \crk F = p}} 
\Ind_{W_F}^W\left(\OS^{W_F}_{M^F\!, \,p}\right)\in \Rep(W).$$
Since $C_\bullet(M)$ has trivial homology, its Euler characteristic is zero.  This proves the lemma.
\end{proof}

The following lemma is an equivariant version of the statement that $\sum_{F\in L}\chi_{M^F}(t) = t^{\rk M}$.

\begin{lemma}\label{counting}
For any equivariant matroid $W\curvearrowright M$,
$$\sum_{[F]\in L/W} \Ind_{W_F}^W\left(H_{M^F}^{W_F}(t)\right) = \sum_{F\in L} \frac{|W_F|}{|W|}\Ind_{W_F}^W\left(H_{M^F}^{W_F}(t)\right) = t^{\rk M}\tau_W,$$
where $\tau_W$ is the trivial representation of $W$.
\end{lemma}

\begin{proof}
Applying Lemma \ref{Brieskorn} to $W_F\curvearrowright M^F$, we have
\begin{eqnarray*}&& 
\sum_{F\in L} \frac{|W_F|}{|W|} \Ind_{W_F}^W\left(H_{M^F}^{W_F}(t)\right)\\
&=& \sum_{F\leq G}\frac{|W_{FG}|}{|W|}(-1)^{\rk G - \rk F}t^{\crk G} \Ind_{W_{FG}}^W\left(\OS_{M^F_G,\,\rk G - \rk F}^{W_F^G}\right)\\
&=& \sum_{G\in L} \frac{|W_G|}{|W|}(-1)^{\rk G} t^{\crk G}\Ind_{W_G}^W\left(\sum_{F\leq G}\frac{|W_{FG}|}{|W_G|}(-1)^{\rk F}\Ind_{W_{FG}}^{W_G}\left(\OS_{M^F_G, \,\rk G - \rk F}^{W_F^G}\right)\right).
\end{eqnarray*}
Applying Lemma \ref{dual} to $W_G\curvearrowright M_G$, we have 
$$\sum_{F\leq G}\frac{|W_{FG}|}{|W_G|}(-1)^{\rk F}\Ind_{W_{FG}}^{W_G}\left(\OS_{M^F_G, \,\rk G - \rk F}^{W_F^G}\right) = 0$$
unless $G$ is equal to the unique flat of rank 0, in which case it is equal to $\tau_{W}$.  
\end{proof}

\begin{remark}
Suppose that $M$ is $W$-equivariantly realizable by a complex linear space $V$.  More precisely, suppose that we are given
a linear subspace $V\subset \C^\mathcal{I}$, preserved by the action of $W$, such that a subset $B\subset \cI$ is a basis for $M$
if and only if the projection of $V$ onto $\C^B$ is an isomorphism.
In this case, Lemma \ref{counting} has a nice geometric interpretation.
The right-hand side of the equation is clearly isomorphic to the compactly supported cohomology of $V$. It is possible to compute
this cohomology via a spectral sequence whose $E_1$ page consists of the compactly supported cohomology groups of the various
strata.  By comparing the mixed Hodge structures on the various groups, we can conclude that this spectral sequence degenerates
at the $E_2$ page, which is given by the left-hand side of the equation.
\end{remark}

The following lemma is an equivariant version of \cite[Lemma 2.1]{EPW}.

\begin{lemma}\label{whatevernickwantstorenamethis}
For any equivariant matroid $W \curvearrowright M$ of positive rank, $$\sum_{[F]\in L/W} \Ind_{W_F}^{W} \left( t^{\rk F} H_{M_F}^{W_F} (t^{-1}) \otimes H_{M^F}^{W_F}(t)\right) 
= \sum_{F\in L}\frac{|W_F|}{|W|} \Ind_{W_F}^{W} \left( t^{\rk F} H_{M_F}^{W_F} (t^{-1}) \otimes H_{M^F}^{W_F}(t)\right) 
=0.$$
\end{lemma}

\begin{proof}
Applying Lemma \ref{Brieskorn} to $W_F\curvearrowright M_F$, we have
\begin{eqnarray*}
&& 
\sum_{F\in L} \frac{|W_F|}{|W|}\Ind_{W_F}^{W} \left(t^{\rk F} H_{M_F}^{W_F} (t^{-1}) \otimes H_{M^F}^{W_F}(t)\right)\\
&=& \sum_{E\leq F}\frac{|W_{EF}|}{|W|}(-1)^{\rk E}t^{\rk E}\Ind_{W_F}^{W} \left(\Ind_{W_{EF}}^{W_F}
\left(\OS_{M_E,\,\rk E}^{W_{EF}}\right)\otimes H_{M^F}^{W_F}(t)\right)\\
&=& \sum_{E\leq F}\frac{|W_{EF}|}{|W|}(-1)^{\rk E}t^{\rk E}\Ind_{W_{EF}}^{W} 
\left(\OS_{M_E,\,\rk E}^{W_{EF}}\otimes H_{M^F}^{W_{EF}}(t)\right)\\
&=& \sum_{E\leq F}\frac{|W_{EF}|}{|W|}(-1)^{\rk E}t^{\rk E}\Ind_{W_{E}}^{W} 
\left(\OS_{M_E,\,\rk E}^{W_{E}}\otimes \Ind_{W_{EF}}^{W_E}\left(H_{M^F}^{W_{EF}}(t)\right)\right)\\
&=& \sum_{E\in L}\frac{|W_{E}|}{|W|}(-1)^{\rk E}t^{\rk E}\Ind_{W_{E}}^{W} 
\left(\OS_{M_E,\,\rk E}^{W_{E}}\otimes \sum_{E\leq F}\frac{|W_{EF}|}{|W_{E}|}\Ind_{W_{EF}}^{W_E}\left(H_{M^F}^{W_{EF}}(t)\right)\right).
\end{eqnarray*}
Applying Lemma \ref{counting} to $W_E\curvearrowright M^E$, this becomes
\begin{eqnarray*}
&& \sum_{E\in L}\frac{|W_{E}|}{|W|}(-1)^{\rk E}t^{\rk E}\Ind_{W_{E}}^{W} 
\left(\OS_{M_E,\,\rk E}^{W_{E}}\otimes t^{\crk E}\tau_{W_E}\right)\\
&=& t^{\rk M}\sum_{E\in L}\frac{|W_{E}|}{|W|}(-1)^{\rk E}\Ind_{W_{E}}^{W} 
\left(\OS_{M_E,\,\rk E}^{W_{E}}\right)\\
&=& t^{\rk M} H_M^W(1),
\end{eqnarray*}
which vanishes by Lemma \ref{one}.
\end{proof}

\subsection{The equivariant Kazhdan-Lusztig polynomial}
Now that we have established some basic properties of the equivariant characteristic polynomial, we are ready
to define the equivariant Kazhdan-Lusztig polynomial.  In the non-equivariant case, this polynomial is defined in \cite[Theorem 2.2]{EPW}.
The following theorem is a categorical version of that result.

\begin{theorem}\label{def}
There is a unique way to assign to each equivariant matroid $W \curvearrowright M$ an element $P^W_M(t)\in\grVRep(W)$,
called the {\bf equivariant Kazhdan-Lusztig polynomial},
such that the following conditions are satisfied:
\begin{enumerate}
\item If $\rk M = 0$, then $P^W_M(t)$ is equal to the trivial representation in degree 0.
\item If $\rk M > 0$, then $\deg P^W_M(t) < \tfrac{1}{2}\rk M$.
\item For every $M$, $\displaystyle t^{\rk M} P^W_M(t^{-1}) = \sum_{[F] \in L/W}\Ind_{W_F}^W\left(\OSc^{W_F}_{M_F}(t) 
\otimes P^{W_F}_{M^F}(t)\right).$
\item Given a homomorphism $\varphi:W'\to W$, $P^{W'}_M(t) = \varphi^* P^W_M(t)$.
\end{enumerate}
\end{theorem}

\begin{proof}
Let $M$ be a matroid of positive rank. We may assume inductively that $P_{M'}^{W'}(t)$ has been defined for every 
matroid $M'$ of rank strictly smaller than $\rk M$ and every group $W'$ acting on $M'$. In particular, $P^W_{M^F}(t)$ has been defined for all $\emptyset \neq F \in L(M)$. Let
$$R^W_M(t) \;\;:= \sum_{\emptyset \not = [F] \in L/W} \Ind_{W_F}^W \left( H_{M_F}^{W_F} (t) \otimes P_{M^F}^{W_F}(t) \right);$$
then item 3 says that we want
$$t^{\rk M}P^W_M(t^{-1}) - P_M^W(t) = R^W_M(t).$$
It is clear that there can be at most one element $P^W_M(t)\in\grVRep(W)$ of degree strictly less than $\frac 12 \rk M$ satisfying this condition. The existence of such a polynomial is equivalent to the statement that
$$t^{\rk M} R_M^W(t^{-1}) = - R^W_M(t),$$ so this is what we need to prove.
We have
\begin{align*}
t^{\rk M}R^W_M(t^{-1}) &= t^{\rk M} \sum_{\emptyset \neq [F] \in L/W} \Ind_{W_F}^W \left( H_{M_F}^{W_F}(t^{-1}) \otimes P_{M^F}^{W_F}(t^{-1}) \right)\\
&= \sum_{\emptyset \neq [F] \in L/W} \Ind_{W_F}^W \left( t^{\rk F} H_{M_F}^{W_F}(t^{-1}) \otimes t^{\rk M^F} P_{M^F}^{W_F}(t^{-1}) \right)\\
&=\sum_{\emptyset \neq [F] \in L/W} \Ind_{W_F}^W \left( t^{\rk F} H_{M_F}^{W_F}(t^{-1}) \otimes \sum_{[G] \in L^F/W_{F}} \Ind^{W_F}_{W_{FG}} \left( H^{W_{FG}}_{M^F_G} (t) \otimes P^{W_{FG}}_{M^G} (t) \right) \right)\\
&= \sum_{\emptyset \neq F \in L} \frac{|W_F|}{|W|} \Ind_{W_F}^W \left( t^{\rk F} H_{M_F}^{W_F}(t^{-1}) \otimes \sum_{G \in L^F} \frac{|W_{FG}|}{|W_F|}\Ind^{W_F}_{W_{FG}} \left( H^{W_{FG}}_{M^F_G} (t) \otimes P^{W_{FG}}_{M^G} (t) \right) \right)\\
&=
\sum_{\emptyset\neq F\leq G} \frac{|W_{FG}|}{|W|} \Ind_{W_F}^W \left( t^{\rk F} H_{M_F}^{W_F}(t^{-1}) \otimes  \Ind^{W_F}_{W_{FG}} \left( H^{W_{FG}}_{M^F_G} (t) \otimes P^{W_{FG}}_{M^G} (t) \right) \right)\\
&= \sum_{\emptyset \neq F \leq G} \frac{|W_{FG}|}{|W|} \Ind_{W_{FG}}^W \left( t^{\rk F} H_{M_F}^{W_{FG}}(t^{-1}) \otimes H^{W_{FG}}_{M^F_G} (t) \otimes P^{W_{FG}}_{M^G} (t) \right)\\
&= \sum_{G\neq\emptyset}\sum_{F \leq G} \frac{|W_{FG}|}{|W|} \Ind_{W_{FG}}^W \left( t^{\rk F} H_{M_F}^{W_{FG}}(t^{-1}) \otimes H^{W_{FG}}_{M^F_G} (t) \otimes P^{W_{FG}}_{M^G} (t) \right) - R_M^W(t).
\end{align*}
Thus it will suffice to show that, for any flat $G\neq\emptyset$,
$$\sum_{F \leq G} \frac{|W_{FG}|}{|W|} \Ind_{W_{FG}}^W \left( t^{\rk F} H_{M_F}^{W_{FG}}(t^{-1}) \otimes H^{W_{FG}}_{M^F_G} (t) \otimes P^{W_{FG}}_{M^G} (t) \right)=0.$$
Indeed, fixing a flat $G\neq \emptyset$, we have
\begin{eqnarray*}
&& \sum_{F \leq G} \frac{|W_{FG}|}{|W|} \Ind_{W_{FG}}^W \left( t^{\rk F} H_{M_F}^{W_{FG}}(t^{-1}) \otimes H^{W_{FG}}_{M^F_G} (t) \otimes P^{W_{FG}}_{M^G} (t) \right)\\
&=& 
\sum_{F\leq G} \frac{|W_{FG}|}{|W|} \Ind_{W_G}^W \left( P_{M^G}^{W_G}(t) \otimes \Ind_{W_{FG}}^{W_G} \left(  t^{\rk F} H_{M_F}^{W_{FG}}(t^{-1}) \otimes H^{W_{FG}}_{M^F_G}(t) \right) \right)\\
&=& \frac{|W_G|}{|W|} \Ind_{W_G}^W   \left( P_{M^G}^{W_G}(t) \otimes \sum_{F \in L_G} \frac{|W_{FG}|}{|W_G|} \Ind_{W_{FG}}^{W_G} \left(  t^{\rk F} H_{M_F}^{W_{FG}}(t^{-1}) \otimes H^{W_{FG}}_{M^F_G}(t)  \right) \right).
\end{eqnarray*}
Lemma \ref{whatevernickwantstorenamethis}, applied to $W_G \curvearrowright M_G$, says that the internal sum is zero,
as desired.
\end{proof}

The following result follows immediately by looking at the coefficient of $t^{\rk M - i}$ in item 3 above.

\begin{proposition}\label{coef}
Let $C^W_{M,i}\in \VRep(W)$ be the coefficient of $t^i$ in $P^W_M(t)$.  If $i < \tfrac{1}{2}\rk M$, then
$$C^W_{M,i} = \sum_{\substack{[F]\in L/W\\ 0\leq j\leq \rk F}} (-1)^j\; \Ind_{W_F}^W\!\Big(\OS^{W_F}_{M_F,j}\otimes\; C^{W_F}_{M^F\!,\, \,\crk F - i + j}\Big).$$
\end{proposition}

\begin{proof}
This follows immediately by looking at the coefficient of $t^{\rk M - i}$ on both sides of the equation in Theorem \ref{def}(3).
\end{proof}

\begin{corollary}\label{zero-one}
For any equivariant matroid $W\curvearrowright M$, 
$$C^W_{M,0} = \tau_W\and C^W_{M,1}\; = \sum_{\substack{[F]\in L/W\\ \crk F = 1}} \Ind_{W_F}^W(\tau_{W_F}) - OS^W_{M,1}.$$
\end{corollary}

\begin{proof}
We apply Proposition \ref{coef}.  When $i=0$, $C^{W_F}_{M^F\!,\, \,\crk F - i + j}\neq 0$ only if $j=0-\crk F$.
The proposition then says that $C^W_{M,0}$ is equal to $C^W_{M^F,0}$, where $F$ is the unique flat of corank 0.
By parts 1 and 4 of Theorem \ref{def}, this is equal to $\tau_W$.

When $i=1$, we have $C^{W_F}_{M^F\!,\, \,\crk F - i + j}\neq 0$ only if $j=0$ and $\crk F=1$ or $j=1$ and $\crk F = 0$.
The first case gives us a contribution of
$$\sum_{\substack{[F]\in L/W\\ \crk F = 1}} \Ind_{W_F}^W(\tau_{W_F})$$
(the permutation representation given by the action of $W$ on the set of corank 1 flats)
and the second case gives us a contribution of $-OS^W_{M,1}$.
\end{proof}

\begin{remark}
By taking dimensions of the representations in Corollary \ref{zero-one}, we recover Propositions 2.11 and 2.12 of \cite{EPW}.
\end{remark}

Suppose that $M$ is $W$-equivariantly realizable by a complex linear space $V\subset \C^{\cI}$.  
Let $X_V$ be the {\bf reciprocal plane}, which is defined as follows:
$$X_V := \overline{\{z\in (\C^\times)^\mathcal{I}\mid z^{-1}\in V\}} \subset \C^\cI.$$
The action of $W$ on $\cI$ induces an action on $X_V$.  The following corollary is an equivariant
version of \cite[Proposition 3.12]{EPW}.

\begin{corollary}\label{rep}
If $M$ is $W$-equivariantly realizable by a linear subspace $V\subset\C^\cI$, then we have $C^W_{M,i} = I\! H^{2i}(X_V; \C)\in\grRep(W)$.
\end{corollary}

\begin{proof}
This follows from Proposition \ref{coef} and \cite[Remark 3.6]{PWY}.
\end{proof}

By definition, the coefficients $C^W_{M,i}$ are virtual representations of $W$.  When $M$ is $W$-equivariantly realizable, however,
Corollary \ref{rep} implies that they are honest representations.  We conjecture that this is always the case, even if $M$ is not
equivariantly realizable.

\begin{conjecture}\label{positivity}
For any equivariant matroid $W \curvearrowright M$, $P^W_M(t)\in\grRep(W)$.
\end{conjecture}

\begin{remark}
When $W$ is the trivial group, Conjecture \ref{positivity} says that the coefficients of the ordinary Kazhdan-Lusztig polynomial
are natural numbers rather than just integers.  This conjecture appeared in \cite[Conjecture 2.3]{EPW}.  We note, however,
that it is much easier to construct non-realizable examples of equivariant matroids than it is to construct non-realizable examples
of ordinary matroids.  For example, let $U_{m,d}$ be the 
the uniform matroid of rank $d$ on $m+d$ elements.  This matroid is always realizable.  However, it has an action of the symmetric
group $S_{m+d}$, and it is equivariantly realizable if and only if $d\in\{0,1\}$ or $m\in\{0,1\}$.
In the following section, we will prove Conjecture \ref{positivity} for arbitrary uniform matroids.
\end{remark}

\section{Uniform matroids}\label{sec:uniform}
Let $U_{m,d}$ be the 
the uniform matroid of rank $d$ on $m+d$ elements, which admits an action of the symmetric group $S_{m+d}$.
Let $$H_{m,d}(t) := H_{U_{m,d}}^{S_{m+d}}(t),\qquad P_{m,d}(t) := P_{U_{m,d}}^{S_{m+d}}(t),\and
C_{m,d,i} := C_{U_{m,d},i}^{S_{m+d}}.$$
For any partition $\la$ of $m+d$, let $V[\la]$ be the irreducible representation of $S_{m+d}$ indexed by $\la$.
The purpose of this section is to prove the following result.

\begin{theorem}\label{uniform}For all $i>0$,
$${C_{m,d,i}\; 
\;\; = \sum_{b=1}^{\min(m,d-2i)} V[d+m-2i-b+1,b+1,2^{i-1}]\;\in\;\Rep(S_{m+d}).}$$
\end{theorem}

\begin{corollary}
Conjecture \ref{positivity} holds for all uniform matroids.
\end{corollary}

\begin{remark}
When $m=1$, Theorem \ref{uniform} specializes to the main result of \cite{PWY}.
\end{remark}

\begin{remark}
One can use Theorem \ref{uniform}, along with the hook length formula for the dimension of $V_\la$,
to compute the coefficients of the ordinary Kazhdan-Lusztig polynomial of $U_{m,d}$.  
(The formula is unenlightening, so we will not reproduce it here.)  This is a computation
that we were unable to do in \cite{EPW}; see Section 2.4 and the appendix of that paper.  Indeed, we still know of no way to compute
these numbers that does not go through Theorem \ref{uniform}.
\end{remark}

\begin{remark}
One immediate consequence of Theorem \ref{uniform} is that,
for any triple $(m,d,i)$, we have $C_{m,d,i} = C_{d-2i,m+2i,i}$.  
This was first observed empirically in the non-equivariant setting by Max Wakefield, 
based on computer calculations.
We still have no philosophical explanation for which this symmetry should exist.
\end{remark}

\begin{remark}\label{stable}
Another consequence of Theorem \ref{uniform} is that
the sequence $(C_{m,d,i})_{d\in\N}$ is uniformly representation stable in the sense of Church and Farb \cite[Definition 2.3]{CF},
with the stable range beginning at $d=m+2i$.
\end{remark}

\subsection{Translating to symmetric functions}
The {\bf Frobenius characteristic} is an isomorphism of vector spaces
$$\ch:\grVRep(S_n)\to \Lambda_n[t],$$
where $\Lambda_n$ is the space of symmetric functions of degree $n$
\cite[Section I.7]{Macdonald}.
It has the property that, given two graded virtual representations $V_1\in \grVRep(S_{n_1})$ and $V_2\in \grVRep(S_{n_2})$, we have
$$\ch \Ind_{S_{n_1}\times S_{n_2}}^{S_{n_1+n_2}}\!\big(V_1\boxtimes V_2\big) = \ch(V_1)\ch(V_2).$$
Let $$\cH_{m,d}(t) := \ch H_{m,d}(t), \qquad \cP_{m,d}(t) := \ch P_{m,d}(t), \and \cC_{m,d,i} := \ch C_{m,d,i}.$$
Applying the Frobenious characteristic to the equation in Theorem \ref{def}(3) (and applying Corollary \ref{zero-one}), 
we obtain the statement
\begin{equation}\label{md}
t^d \cP_{m,d}(t^{-1}) = \cH_{m,d}(t) + \sum_{k=1}^{d}
\cH_{0,d-k}(t) \cP_{m,k}(t).\end{equation}

\subsection{Generating functions}
In this section we will work in the ring $\Lambda[[t,u,x]]$ of completed
symmetric functions with coefficients in the ring of formal power series in $t$, $u$, and $x$.
Let $$\cH(t,u,x) := \sum_{d=1}^\infty\sum_{m=0}^\infty \cH_{m,d}(t)u^dx^m
\and \cP(t,u,x) := \sum_{d=1}^\infty\sum_{m=0}^\infty \cP_{m,d}(t)u^dx^m.$$
Then Equation \eqref{md} for all values of $m$ and $d$ is equivalent to the generating function equation
\begin{equation}\label{gen}
\cP(t^{-1},tu,x) = \cH(t,u,x) + \big(1 + \cH(t,u,0)\big)\cP(t,u,x).
\end{equation}

\begin{remark}\label{eq-to-exp}
Once we have Equation \eqref{gen}, we obtain for free the corresponding functional equation involving the 
(non-equivariant) exponential generating functions.
Let $$H(t,u,x) := \sum_{d=1}^\infty\sum_{m=0}^\infty \dim H_{m,d}(t) \frac{u^d x^m}{(d+m)!} = \sum_{d=1}^\infty\sum_{m=0}^\infty \chi_{U_{m,d}}(t) 
\frac{u^d x^m}{(d+m)!}$$
and $$P(t,u,x) := \sum_{d=1}^\infty\sum_{m=0}^\infty \dim P_{m,d}(t) \frac{u^d x^m}{(d+m)!} = \sum_{d=1}^\infty\sum_{m=0}^\infty P_{U_{m,d}}(t) 
\frac{u^d x^m}{(d+m)!}.$$
Then we have \begin{equation}\label{uniform exp}
P(t^{-1},tu,x) = H(t,u,x) + \big(1 + H(t,u,0)\big)P(t,u,x).
\end{equation}
Equation \eqref{uniform exp} follows from Equation \eqref{gen} using the following easy observation:
Let $V_i$ be a representation of $S_{n_i}$ for $i\in\{1,2\}$,
and let $V = \Ind_{S_{n_1}\times S_{n_2}}^{S_{n_1+n_2}}\!\left(V_1\boxtimes V_2\right)$.
Then $\frac{\dim V}{(n_1+n_2)!} = \frac{\dim V_1}{n_1!}\cdot\frac{\dim V_2}{n_2!}$.
\end{remark}

The remainder of this section is devoted to deriving explicit expressions for $\cH(t,u,x)$ and $H(t,u,x)$.
Given a partition $\la$ of $n$, let $s[\la] := \ch V[\la]$ be the Schur function associated with $\la$.
Let $$s(t) := \sum_{n=0}^\infty t^n s[n],$$
and recall the well-known identity $$s(-t)^{-1} = \sum_{n=0}^\infty t^n s[1^n].$$

\begin{lemma}\label{weird}
$\displaystyle \sum_{e=0}^\infty\sum_{m=0}^\infty t^{e}u^{m} s[m+1,1^{e}] = \frac{1}{t+u}\left(-1 + \frac{s(u)}{s(-t)}\right).$
\end{lemma}

\begin{proof}
We have
\begin{eqnarray*}-1 + \frac{s(u)}{s(-t)} &=&  -1 + \left(\sum_{j=0}^\infty t^{j} s[1^{j}]\right)\left(\sum_{k=0}^\infty u^{k} s[k]\right)\\
&=& -1 + \sum_{j=0}^\infty\sum_{k=0}^\infty t^ju^k s[1^{j}]s[k]\\
&=& \sum_{j=1}^\infty\sum_{k=0}^\infty t^ju^ks[k+1,1^{j-1}] + \sum_{j=0}^\infty\sum_{k=1}^\infty t^ju^ks[k,1^j],
\end{eqnarray*}
where the last equality follows from the Pieri rule.  Next, observe that
$$\sum_{j=1}^\infty\sum_{k=0}^\infty t^ju^k s[k+1,1^{j-1}] = t\sum_{e=0}^\infty\sum_{m=0}^\infty t^{e}u^{m} s[m+1,1^{e}]$$
and $$\sum_{j=0}^\infty\sum_{k=1}^\infty t^ju^k s[k,1^j] = u\sum_{e=0}^\infty\sum_{m=0}^\infty t^{e}u^{m} s[m+1,1^{e}].$$
Adding them together and dividing by $t+u$, we obtain the desired equation.
\end{proof}

\begin{proposition}\label{htux}
We have $$\cH(t,u,x) = \frac{u}{u-x}\left(-1 + \frac{s(x)}{s(u)}\right) + \frac{tu}{tu-x}\left(\frac{s(tu)}{s(u)} - \frac{s(x)}{s(u)}\right)$$
and $$1 + \cH(t,u,0) = \frac{s(tu)}{s(u)}.$$
\end{proposition}

\begin{proof}
If $i<d$, then 
$$\OS^{S_{m+d}}_{U_{m,d},i} = \wedge^i \C^{m+d} = V[m+d-i,1^i] + V[m+d-i+1,1^{i-1}].$$
Applying the Frobenius characteristic, we have
$$\ch \OS^{S_{m+d}}_{U_{m,d},i} = s[m+d-i,1^i] + s[m+d-i+1,1^{i-1}] = s[1^i]s[m+d-i],$$
where the second equation follows from the Pieri rule.
We also have
$$\OS^{S_{m+d}}_{U_{m,d},d} = V[m+1,1^{d-1}],$$
and therefore 
$$\ch \OS^{S_{m+d}}_{U_{m,d},d} = s[m+1,1^{d-1}].$$
By definition,
$$\cH(t,u,x) = \sum_{d=1}^\infty\sum_{m=0}^\infty\sum_{i=0}^\infty (-t^{-1})^i (tu)^d x^m \ch \OS^{S_{m+d}}_{U_{m,d},i}.$$
Letting $k=d-i$, this tells us that
\begin{eqnarray*}
\cH(t,u,x) 
&=& \sum_{d=1}^\infty\sum_{m=0}^\infty (-u)^d x^m \ch \OS^{S_{m+d}}_{U_{m,d},d}
+ \sum_{k=1}^\infty\sum_{m=0}^\infty\sum_{i=0}^\infty (-t^{-1})^i (tu)^{i+k} x^m \ch \OS^{S_{m+i+k}}_{U_{m,i+k},i}\\
&=& \sum_{d=1}^\infty\sum_{m=0}^\infty (-u)^d x^m s[m+1,1^{d-1}]
+ \sum_{k=1}^\infty\sum_{m=0}^\infty\sum_{i=0}^\infty (-u)^i (tu)^k x^m s[1^i]s[m+k]\\
&=& \sum_{d=1}^\infty\sum_{m=0}^\infty (-u)^d x^m s[m+1,1^{d-1}]
+ \left(\sum_{i=0}^\infty (-u)^i s[1^i]\right)\left(\sum_{k=1}^\infty\sum_{m=0}^\infty (tu)^k x^m s[m+k]\right)\\
&=& \frac{u}{u-x}\left(-1 + \frac{s(x)}{s(u)}\right)
+ \frac{tu}{tu-x}\cdot\frac{s(tu)-s(x)}{s(u)},
\end{eqnarray*}
where the last equation follows from
Lemma \ref{weird}.  The second statement is obtained from the first by setting $x$ equal to zero.
\end{proof}

The following Corollary follows immediately from Proposition \ref{htux} as in Remark \ref{eq-to-exp}.
We use the fact that $s(t)$ is the Frobenius characteristic of the trivial representation in every degree,
so the exponential generating function for its dimensions is $e^t$.

\begin{corollary}\label{exp-htux}
We have $$H(t,u,x) = \frac{u}{u-x}\left(-1 + e^{x-u}\right) + \frac{tu}{tu-x}\left(e^{tu-u} - e^{x-u}\right)$$
and $$1 + H(t,u,0) = e^{tu-u}.$$
\end{corollary} 

Proposition \ref{htux} combines with Equation \eqref{gen} to tell us that 
$$\cP(t^{-1},tu,x) = \frac{u}{u-x}\left(-1 + \frac{s(x)}{s(u)}\right) + \frac{tu}{tu-x}\left(\frac{s(tu)}{s(u)} - \frac{s(x)}{s(u)}\right) + \cP(t,u,x) \frac{s(tu)}{s(u)}.$$
Rearranging terms, this is equivalent to the equation
\begin{equation}\label{sym}
\left(\frac{u}{u-x}+\cP(t^{-1},tu,x)\right)s(u)
- \frac{u}{u-x}s(x) = 
\left(\frac{tu}{tu-x} + \cP(t,u,x)\right) s(tu)
- \frac{tu}{tu-x}s(x).\end{equation}
Let \begin{eqnarray*}\cR(t,u,x) &:=& \left(\frac{tu}{tu-x} + \cP(t,u,x)\right)s(tu)
- \frac{tu}{tu-x} s(x)\\
&=& \left(\frac{tu}{tu-x}+ \cP(t,u,x)\right)\sum_{n=0}^\infty (tu)^n s[n] - \frac{tu}{tu-x}\sum_{n=0}^\infty x^n s[n]\end{eqnarray*}
 be the expression on the right-hand side of Equation \eqref{sym}.
Then Equation \eqref{sym} becomes $$\cR(t^{-1},tu,x) = \cR(t,u,x).$$
The results of this section can be summarized as follows.

\begin{proposition}\label{reformulation}
The element $\cP(t,u,x)\in \Lambda[[t,u,x]]$ is uniquely characterized by the following properties:
\begin{itemize}
\item $\cP(t,0,x) = 0$,
\item the coefficient of $t^iu^d$ in $\cP(t,u,x)$ is zero if $2i\geq d$,
\item $\cR(t^{-1},tu,x) = \cR(t,u,x)$, where $\cR(t,u,x)$ is defined above.
\end{itemize}
\end{proposition}

\begin{remark}
It is interesting to observe exactly what our manipulations of generating functions has bought us.
The straightforward apporach to proving Theorem \ref{uniform} would have been to apply Proposition \ref{coef}
and proceed by induction on $d$.  This works in theory, but it involves repeated applications of the Littlewood-Richardson
rule for hooks, and the combinatorics very quickly gets out of hand.
Instead, we will prove Theorem \ref{uniform} by taking our ``guess" for $\cP(t,u,x)$ and verifying the equation 
$\cR(t^{-1},tu,x) = \cR(t,u,x)$.  From the definition of $\cR(t,u,x)$, we see that this will involve repeated
applications of the Pieri rule, which is much simpler than the general Littlewood-Richardson rule.  This simplification
is exactly what makes our proof possible.
\end{remark}

\subsection{Proving the theorem}
We are now ready to prove Theorem \ref{uniform}.
Let $$\cP'(t,u,x) := \sum_{d=1}^\infty u^ds[d] + \sum_{i=0}^\infty\sum_{d=1}^\infty\sum_{m=0}^\infty t^iu^dx^m \sum_{b=1}^{\min(m,d-2i)} s[d+m-2i-b+1,b+1,2^{i-1}].$$
Here and throughout this section we adopt the notational convention that $$s[a,2,2^{-1}] = s[a] \and s[a,b+1,2^{-1}] = 0\;\;\text{if $b>1$};$$
in particular, the coefficient of $t^0u^dx^m$ in $\cP'(t,u,x)$ is equal to $s[d+m]$ for any $d\geq 1$ and $m\geq 0$.
Let $$\cR'(t,u,x) := \left(\frac{tu}{tu-x}+ \cP'(t,u,x)\right)\sum_{n=0}^\infty (tu)^n s[n] - \frac{tu}{tu-x}\sum_{n=0}^\infty x^n s[n].$$
By Proposition \ref{reformulation}, Theorem \ref{uniform} is equivalent to the statement that 
$\cR'(t^{-1},tu,x) = \cR'(t,u,x)$.

The coefficient of $t^i u^d x^m$ in $\cR'(t,u,x)$ is equal to 
\begin{eqnarray*}
\sum_{b=1}^m\sum_{k=b-d+2i}^i s[k] s[d+k+m-2i-b+1,b+1,2^{i-k-1}]
\;\;&+&\;\; \begin{cases}
s[i]s[d-i]\;\;\;\text{if $m=0<d-i$}\\
0\;\;\;\text{otherwise}
\end{cases}\\
&+&\;\; \begin{cases}
s[m+d]\;\;\;\text{if $i=d>0$}\\
0\;\;\;\text{otherwise.}
\end{cases}\end{eqnarray*}
Thus we need to show that this expression is invariant under the substitution $i\leftrightarrow d-i$.
If $d=0$ or $i>d$, the expression is equal to zero.  If $d>0$ and $i=0$ or $i=d$, then the expression is equal to $s[d+m]$.
Thus, we may assume that $0<i<d$.
If $m=0$, the expression is equal to $s[i]s[d-i]$, which is clearly invariant.  Thus, we may further assume that $m\neq 0$, which means
that we can restrict our attention to the expression
$$\sum_{b=1}^m\sum_{k=b-d+2i}^i s[k] s[d+k+m-2i-b+1,b+1,2^{i-k-1}].$$
Letting $r=d-2i$, we can rewrite this expression as $$\Psi(i,r,m) := \sum_{b=1}^m\sum_{k=b-r}^i s[k] s[r+k+m-b+1,b+1,2^{i-k-1}],$$
and we want to show that it is equal to
$$\Phi(i,r,m) := \Psi(i+r,-r,m) = \sum_{b=1}^m\sum_{j=b}^{i} s[j+r] s[j+m-b+1,b+1,2^{i-j-1}].$$

The Pieri rule tells us that any Schur function appearing $\Psi(i,r,m)$
must be of the form $s[A,B,C,2^D]$ or $s[A,B,C,2^{D},1]$, 
where we continue to adhere to our notational convention:
$$s[A,B,2,2^{-1}] = s[A,B],\qquad s[A,2,2,2^{-2}] = s[A],$$
and so on.  

Let us focus on the case of $s[A,B,C,2^D]$ with $D\geq 0$ and $A+B+C+2D = 2i+r+m$.
The Schur function $s[A,B,C,2^D]$ can appear in the $k$ summand of $\Psi(i,r,m)$ when $k=i-D-2$ or $k=i-D-1$.
In each of these summands, the number of times that $s[A,B,C,2^D]$ is equal to the number of values
of $b\leq \min(m,k+r)$ satisfying the inequalities 
\begin{equation*}\label{ineqs}
A \geq r+k+m-b+1\geq B\geq b+1 \geq C \geq 2,
\end{equation*}
which ensure that the partitions $[r+k+m-b+1,b+1,2^{i-k-1}]$ and $[A,B,C,2^D]$ interlace.
More precisely, 
let $$\varepsilon_1 := \min(m,r+i-D-1+m-B,B-1,i-D-2+r),\qquad \varepsilon_2 := \min(m,r+i-D+m-B,B-1,i-D-1+r),$$
$$\Upsilon_1 := \max(C-1,r+i-D-1+m-A),\and \Upsilon_2 := \max(C-1,r+i-D+m-A).$$
Then the coefficient
of $s[A,B,C,2^D]$ in $\Psi(i,r,m)$ is equal to
$$\max(0, \varepsilon_1-\Upsilon_1+1) + \max(0, \varepsilon_2-\Upsilon_2+1),$$
where the first summand represents the number of possible values for $b$ when $k=i-D-2$, and the second summand
represents the number of possible values for $b$ when $k=i-D-1$.

Similarly, let
$$\varepsilon'_1 := \min(m,i-D-1+m-B,B-1,i-D-2),\qquad \varepsilon'_2 := \min(m,i-D+m-B,B-1,i-D-1),$$
$$\Upsilon'_1 := \max(C-1,i-D-1+m-A),\and \Upsilon'_2 := \max(C-1,i-D+m-A).$$
Then the coefficient
of $s[A,B,C,2^D]$ in $\Phi(i,r,m)$ is equal to
$$\max(0, \varepsilon'_1-\Upsilon'_1+1) + \max(0, \varepsilon'_2-\Upsilon'_2+1),$$
where the first summand represents the number of possible values for $b$ when $j=i-D-2$, and the second summand
represents the number of possible values for $b$ when $j=i-D-1$.
Our plan is to show that $$\varepsilon_1-\Upsilon_1 = \varepsilon'_2-\Upsilon'_2
\and
\varepsilon_2-\Upsilon_2 = \varepsilon'_1-\Upsilon'_1,$$
which will tell us that $s[A,B,C,2^D]$ appears with the same coefficient in $\Psi(i,r,m)$ and $\Phi(i,r,m)$.

\begin{lemma}\label{maxiff}  We have 
\begin{enumerate}
\item $\Upsilon_1 = C-1$ if and only if $\varepsilon'_2 = \min(m,B-1)$,
\item $\Upsilon_2 = C-1$ if and only if $\varepsilon'_1 =\min(m,B-1)$,
\item $\Upsilon'_1 = C-1$ if and only if $\varepsilon_2=\min(m,B-1)$,
\item $\Upsilon'_2 = C-1$ if and only if $\varepsilon_1=\min(m,B-1)$.
\end{enumerate}
\end{lemma}

\begin{proof}
We prove the forward direction of Lemma \ref{maxiff}(1) and and note that the other cases are identical. 
Since $\Upsilon_1 = C-1$, we have
\[
C-1 \geq r+i-D-1+m-A = B+C+D-i-1,
\]
which implies that
$i-D-1 \geq B-1$.
Adding $m+1$ to both sides and subtracting $B$ tells us that
$i-D+m-B \geq m$,
hence $\varepsilon'_2=m$ or $B-1$.
\end{proof}

Now consider the expression \begin{equation}\label{diff}\varepsilon_1-\Upsilon_1 - \varepsilon'_2 + \Upsilon'_2;\end{equation}
we will use a case-by-case analysis to prove that this expression is equal to zero.\\\\
{\em Case 1:} $\Upsilon_1 = C-1 = \Upsilon'_2$.  Then $\varepsilon_1 =\min(m,B-1) = \varepsilon'_2$, and \eqref{diff} vanishes.\\\\
{\em Case 2:} $\Upsilon_1 \neq C-1 \neq \Upsilon'_2$.  Then 
$$\Upsilon_1 = r+i-D-1+m-A,\qquad \Upsilon'_2 = i-D+m-A = \Upsilon_1-(r-1),$$
$$\varepsilon'_2 = \min(i-D+m-B,i-D-1),\and \varepsilon_1 = \min(r+i-D-1+m-B,i-D-2+r) = \varepsilon'_2+(r-1),$$
and therefore \eqref{diff} vanishes.\\\\
{\em Case 3:} $\Upsilon_1 = C-1 \neq \Upsilon'_2$.  Then $$\Upsilon'_2 = i-D+m-A,\qquad\varepsilon'_2 = \min(m,B-1),\and
\varepsilon_1 = \min(r+i-D-1+m-B,i-D-2+r).$$
Assume first that $\varepsilon'_2 = m$.  This implies that $\varepsilon_1 = r+i-D-1+m-B$, and therefore that the expression \eqref{diff}
is equal to $r+2i+m - A - B - C - 2D = 0$.  On the other hand, if $\varepsilon'_2 = B-1$, then $\varepsilon_1 = i-D-2+r$, and we reach
the same conclusion.\\\\
{\em Case 4:} $\Upsilon_1 \neq C-1 = \Upsilon'_2$.  This is similar to Case 3.\\\\
We have now shown that the expression \eqref{diff} is equal to zero, and therefore that $\varepsilon_1-\Upsilon_1 = \varepsilon'_2-\Upsilon'_2$.
A similar argument allows us to conclude that $\varepsilon_2-\Upsilon_2 = \varepsilon'_1-\Upsilon'_1$.
This completes the proof that $s[A,B,C,2^D]$ appears with the same coefficient in $\Psi(i,r,m)$ and $\Phi(i,r,m)$.
The other cases, namely Schur functions of the form $s[A,B,C,2^D,1]$ and Schur functions of the form $s[A,B,C,2^D]$ with $D<0$,
can be analyzed in a similar fashion; we leave the details of these cases to the reader.  This completes the proof of Theorem \ref{uniform}.\qed

\excise{
\subsection{Exponential generating functions}\label{sec:exp-braid}
\begin{proposition}\label{Htux}
We have $$H(t,u,x) = \frac{xe^{x-u} - tu e^{u(t-1)}}{x-tu} - \frac{xe^{x-u}-u}{x-u}
\and 1 + H(t,u,0) = e^{u(t-1)}.$$
\end{proposition}

\begin{proof}
We have
\begin{eqnarray*}H(t,u,x) &=& \sum_{m\geq 0}\sum_{d\geq 0}\frac{x^m u^d}{(m+d)!}\sum_{i=0}^d \binom{m+d}{i}(-1)^i\big(t^{d-i}-1\big)\\
&=& \sum_{m\geq 0}\sum_{d\geq 0}\sum_{i=0}^d(-1)^i\frac{x^m u^d(t^{d-i}-1)}{i!(m+d-i)!}\\
&=& \sum_{m\geq 0}\sum_{i\geq 0}\sum_{e\geq 0} (-1)^i\frac{x^m u^{i+e}(t^e-1)}{i!(m+e)!}\\
&=& \sum_{i\geq 0}(-1)^i \frac{u^i}{i!} \sum_{m\geq 0}\sum_{e\geq 0}\frac{x^m u^{e}(t^e-1)}{(m+e)!}\\
&=& e^{-u}\left(\sum_{m\geq 0}\sum_{e\geq 0}\frac{x^m (tu)^{e}}{(m+e)!} - \sum_{m\geq 0}\sum_{e\geq 0}\frac{x^m u^{e}}{(m+e)!}\right)\\
&=& e^{-u}\sum_{k\geq 0} \frac{1}{k!} \sum_{m+e=k} \Big( x^m (tu)^e - x^m u^e\Big)\\
&=& e^{-u}\sum_{k\geq 0} \frac{1}{k!} \left( \frac{x^{k+1} - (tu)^{k+1}}{x-tu} - \frac{x^{k+1} - u^{k+1}}{x-u}\right)\\
&=& e^{-u} \left(\frac{xe^x - tue^{tu}}{x-tu} - \frac{xe^x - ue^{tu}}{x-u}\right)\\
&=& \frac{xe^{x-u} - tu e^{u(t-1)}}{x-tu} - \frac{xe^{x-u}-u}{x-u}.
\end{eqnarray*}
The second statement follows from the first by setting $x$ equal to zero.
\end{proof}
}


\section{Braid matroids}\label{sec:braid}
Let $B_n$ be the braid matroid of rank $n-1$.  Equivalently, $B_n$ is the matroid associated with the complete
graph on $n$ vertices.  The ground set is equal to the set of all 2-element subsets of $[n]$, and the lattice of flats
is equal to the set of set-theoretic partitions of $n$; the rank of a flat is equal to $n$ minus the number of parts of
the partition.  The group $S_n$ acts on $B_n$ in the obvious manner.

Let $$K_n(t) := H_{B_n}^{S_n}(t), \qquad Q_n(t) := P_{B_n}^{S_n}(t), \and D_{n,i} := C_{B_n,i}^{S_n}.$$
For any partition $\la\vdash n$, let $S_\la\subset S_n$ be the stabilizer of a set-theoretic partition of type $\la$.
Then Theorem \ref{def}(3) says
\begin{equation}\label{braid-recursion} t^{n-1}Q_n(t^{-1}) = \sum_{\la\vdash n}\Ind_{S_\la}^{S_n}\left(K_{\la_1}(t)
\otimes\cdots\otimes K_{\la_{\ell(\la)}}(t) \otimes Q_{\ell(\la)}(t)\right).\end{equation}

The polynomial $K_n(t)$ is well understood, going back to Lehrer and Solomon \cite{LS}.
An explicit formula for $\ch K_n(t)$ appears in \cite[Theorem 2.7]{HershReiner}.
We cannot give an explicit formula for $Q_n(t)$ in the same way that we did for uniform matroids, 
but we will decribe the recursion that can be used to compute it and calculate some examples for small $n$.

\begin{remark}
As defined above, $B_n$ is not equivariantly realizable.  However, let $B'_n$ be the analogously defined
matroid with ground set $\cI_n$ equal the set of {\bf ordered} pairs of distinct elements of $[n]$, with $\{(i,j), (j,i)\}$ 
a dependent set.  
Then $B_n$ is the underlying simple matroid of $B'_n$, so they have isomorphic lattices
of flats, the same equivariant characteristic polynomial, and the same equivariant Kazhdan-Lusztig polynomial.
Let $V_n = \C^n/\C_\Delta$, and consider the embedding of $V_n$ into $\C^{\cI_n}$
given by $x_i-x_j$ in the $(i,j)$-coordinate.  This embedding is an $S_n$-equivariant realization of $B'_n$.
It follows from Corollary \ref{rep} that $$Q_n(t) = P_{B_n}^{S_n}(t) = P_{B'_n}^{S_n}(t) \in\grRep(S_n).$$
\end{remark}

\subsection{Generating functions}
Let $K$ be the (virtual, graded) linear species that assigns to the set $[n]$ the (virtual, graded) representation $K_n(t)$ of $S_n$.
Similarly, let $Q$ be the (graded) linear species that assigns to the set $[n]$ the (graded) representation $Q_n(t)$ of $S_n$.
Finally, motivated by Equation \eqref{braid-recursion}, let $\bar Q$ be the 
(graded) linear species that assigns to the set $[n]$ the (graded) representation $t^{n-1}Q_n(t^{-1})$ of $S_n$.
In the language of species, Equation \eqref{braid-recursion} says that $\bar Q = Q\circ P$.

Consider the following six generating functions:
$$K(t,z) := \sum_{n= 1}^\infty \dim K_n(t)\frac{z^n}{n!}, \;\; Q(t,z) := \sum_{n=1}^\infty \dim Q_n(t)\frac{z^n}{n!}, \;\; 
\bar Q(t) := \sum_{n=1}^\infty \dim \bar Q_n(t)\frac{z^n}{n!}\;\;\in\;\;\Q[[t,z]],$$
$$\cK(t) := \sum_{n= 1}^\infty \ch K_n(t), \;\; \cQ(t) := \sum_{n= 1}^\infty \ch Q_n(t),
\;\; \bar\cQ(t) := \sum_{n= 1}^\infty \ch \bar Q_n(t)\;\;\in\;\;\Lambda[[t]].$$
Note that $\dim K_n(t) = \chi_{B_n}(t) = (t-1)\cdots (t-n+1)$, $\dim Q_n(t) = P_{B_n}(t)$,
and $\dim \bar Q_n(t) = t^{n-1}P_{B_n}(t^{-1})$.
Since we have $\bar Q = Q\circ P$, the theory of species tells us that
\begin{equation}\label{braid-gen}\frac 1 t Q(t^{-1}, tz) = \bar Q(t,z) = Q(t,K(t,z))\and \bar\cQ(t) = \cQ(t)\big[\cK(t)\big],\end{equation}
where square brackets denote plethysm.  See, for example, \cite[Proposition 2.1]{Mendez},
which can be extended to virtual species as in \cite{Joyal}.

\begin{remark}
Equation \eqref{braid-gen} can be derived directly from Equation \eqref{braid-recursion}.
The second half of Equation \eqref{braid-gen} follows from the fact that, for any representation $V_i$ of $S_i$ ($i\in\{1,2\}$),
$\ch \Ind_{S_i\wr S_{j}}^{S_{ij}}\!\left(V_j\otimes V_i^{\otimes j}\right)$ is equal to $\ch V_j[\ch V_i]$.
The first half follows then from the second half by computing 
the effect of this induction on the dimension of a (virtual, graded) representation.
The language of species simply provides a tidy formalism for these observations.
\end{remark}

\begin{remark}
The generating functions $K(t,z)$ and $\cK(t)$ can be understood very explicitly.
We have $$(z+1)^{t} = \sum_{n=0}^\infty\binom{t}{n} z^n = 1 + t\cdot \sum_{n=1}^\infty \chi_{B_n}(t) \frac{z^n}{n}
= 1 + t K(t,z).$$
Based on the work of Lehrer and Solomon, Getzler \cite[Equation (2.5)]{Getzler} gives the following analogous formula
for $\cK(t)$:
$$1 + t\cK(t) = \prod_{k=1}^\infty(1+p_k)^{\frac{1}{k}\sum_{d|k}\mu(k/d)t^{d}},$$
where $p_k$ is the $k^\text{th}$ power sum symmetric function.
\end{remark}

\subsection{Linear term}
Though we have no general formula for $Q_n(t)$, we can use Corollary \ref{zero-one} to calculate $D_{n,1}$.

\begin{proposition}\label{linear braid}
When $n\leq 3$, $D_{n,1}=0$.  When $n\geq 4$, we have $$D_{n,1} = V[n]^{\oplus \left\lfloor\frac{n-2}{2}\right\rfloor}  \oplus V[n-1,1]^{\oplus\left\lfloor\frac{n-3}{2}\right\rfloor}
\oplus V[n-2,2]^{\oplus \left\lfloor\frac{n-4}{2}\right\rfloor} \oplus 
\bigoplus_{3\leq i\leq \lfloor n/2\rfloor}V[n-i,i]^{d(n,i)},$$
where $$d(n,i) = 
\begin{cases}
\quad n/2 - i\;\;\;\,\quad\text{if $n$ is even and $i$ is odd},\\
(n+1)/2 - i\;\;\text{if $n$ is odd},\\
(n+2)/2-i\;\;\text{if $n$ and $i$ are both even}.
\end{cases}
$$\end{proposition}

\begin{proof}
Let $W = S_n$.
By Proposition \ref{zero-one}, $$D_{n,1} = \sum_{\substack{[F]\in L/W\\ \crk F = 1}} \Ind_{W_F}^W(\tau_{W_F}) - OS^{S_n}_{B_n,1}.$$
We have $$OS^{S_n}_{B_n,1} = \Sym^2(\C^n) - \C^n = \Sym^2 V[n-1,1] = V[n-2,2] \oplus V[n-1,1] \oplus V[n].$$
Let $F_k$ be the partition of $[n]$ into $[k]$ and it's complement;
then $\{F_1,\ldots, F_{\lfloor n/2\rfloor}\}$ is a complete set of representatives of $S_n$-orbits of corank 1 flats.

Suppose that $n$ is odd.  Then $W_{F_k} = S_k\times S_{n-k}$, so
\begin{eqnarray*}\ch D_{n,1} &=& \sum_{k=1}^{\frac{n-1}{2}} s[k]s[n-k] - \big(s[n-2,2] + s[n-1,1] + s[n]\big)\\
&=& \sum_{k=1}^{\frac{n-1}{2}}\sum_{i=0}^k s[n-i,i] - \big(s[n-2,2] + s[n-1,1] + s[n]\big)\\
&=& \frac{n-3}{2} s[n] +  \frac{n-3}{2} s[n-1,1] + \frac{n-5}{2} s[n-2,2] + \sum_{i=3}^{\frac{n-1}{2}}\left(\frac{n+1}{2}-i\right)s[n-i,i].
\end{eqnarray*}
If $n$ is even, then $W_{F_k} = S_k\times S_{n-k}$ for all $k<n/2$, but $W_{F_{n/2}} = S_{n/2} \wr S_2$.
We therefore have
\begin{eqnarray*}\ch D_{n,1} &=& \sum_{k=1}^{\frac{n-2}{2}} s[k]s[n-k] + s[2]\Big[s[n/2]\Big] - \big(s[n-2,2] + s[n-1,1] + s[n]\big)\\
&=& \frac{n-4}{2} s[n] +  \frac{n-4}{2} s[n-1,1] + \frac{n-6}{2} s[n-2,2] + \sum_{i=3}^{\frac{n-2}{2}}\left(\frac{n}{2}-i\right)s[n-i,i] + s[2]\Big[s[n/2]\Big].
\end{eqnarray*}
We also have $$s[2]\Big[s[n/2]\Big] = \sum_{j=0}^{n/2} s[n-2j, 2j],$$ and the proposition follows.
\end{proof}

\begin{remark}
Proposition \ref{linear braid} implies that the sequence $\{D_{n,1}\}$ of $S_n$-representations
is {\bf not} representation stable in the sense of Church and Farb.
\end{remark}

\subsection{Calculations}\label{sec:calculations}
We conclude our discussion of braid matroids with a calculation of $D_{n,i}$ for $n\leq 9$ and $i\geq 2$
(since $D_{n,0} = V[n]$ by Proposition \ref{zero-one} and $D_{n,1}$ is computed in Proposition \ref{linear braid}).  These computations were 
performed in SAGE \cite{sage}, using Equation \eqref{braid-recursion}.

\begin{eqnarray*}
D_{6,2} &=& V[2,2,2]\oplus V[4,2]\oplus V[6]\\\\
D_{7,2} &=& V[2,2,2,1]\oplus V[3,2,2]\oplus V[4,2,1]^{\oplus 2}\oplus V[4,3]^{\oplus 2}\oplus V[5,2]^{\oplus 2}\oplus V[6,1]^{\oplus 2}\oplus V[7]^{\oplus 2}\\\\
D_{8,2} &=& V[2,2,2,2]\oplus V[3,2,2,1]\oplus V[4,2,1,1]\oplus V[4,2,2]^{\oplus 4}\oplus V[4,3,1]^{\oplus 3}\oplus V[4,4]^{\oplus 4}\\ && \oplus V[5,2,1]^{\oplus 4}\oplus V[5,3]^{\oplus 4}\oplus V[6,1,1]\oplus V[6,2]^{\oplus 7}\oplus V[7,1]^{\oplus 4}\oplus V[8]^{\oplus 4}\\\\
D_{9,2} &=& V[3,2,2,2]\oplus V[4,2,2,1]^{\oplus 3}\oplus V[4,3,1,1]\oplus V[4,3,2]^{\oplus 4}\oplus V[4,4,1]^{\oplus 6}\oplus V[5,2,1,1]^{\oplus 2} \oplus V[5,2,2]^{\oplus 7} \\ &&\oplus V[5,3,1]^{\oplus 7}\oplus V[5,4]^{\oplus 7}\oplus V[6,2,1]^{\oplus 9}\oplus V[6,3]^{\oplus 12}\oplus V[7,1,1]^{\oplus 3}\oplus V[7,2]^{\oplus 12}\oplus V[8,1]^{\oplus 8}\oplus V[9]^{\oplus 6}\\\\
D_{8,3} &=& V[2,2,2,2]\oplus V[3,2,2,1]\oplus V[3,3,1,1]\oplus V[4,1,1,1,1]\oplus V[4,2,1,1]\oplus V[4,2,2]^{\oplus 3}\\ &&\oplus V[4,3,1] \oplus V[4,4]^{\oplus 2}\oplus V[5,2,1]^{\oplus 2}\oplus V[5,3]\oplus V[6,2]^{\oplus 2}\oplus V[7,1]\oplus V[8]\\\\
D_{9,3} &=& V[2,2,2,2,1]\oplus V[3,2,2,1,1]^{\oplus 2}\oplus V[3,2,2,2]^{\oplus 4}\oplus V[3,3,1,1,1]^{\oplus 2}\oplus V[3,3,2,1]^{\oplus 4}\oplus V[3,3,3]^{\oplus 2}\\ &&\oplus V[4,1,1,1,1,1]\oplus V[4,2,1,1,1]^{\oplus 3}\oplus V[4,2,2,1]^{\oplus 10}\oplus V[4,3,1,1]^{\oplus 7}\oplus V[4,3,2]^{\oplus 10}\oplus V[4,4,1]^{\oplus 8}\\ &&\oplus V[5,1,1,1,1]\oplus V[5,2,1,1]^{\oplus 8}\oplus V[5,2,2]^{\oplus 12}\oplus V[5,3,1]^{\oplus 13}\oplus V[5,4]^{\oplus 7}\oplus V[6,1,1,1]^{\oplus 2}\\ &&\oplus V[6,2,1]^{\oplus 12}\oplus V[6,3]^{\oplus 11}\oplus V[7,1,1]^{\oplus 4}\oplus V[7,2]^{\oplus 9}\oplus V[8,1]^{\oplus 5}\oplus V[9]^{\oplus 3}
\end{eqnarray*}

\section{Equivariant log concavity}\label{sec:elc}
Fix a finite group $W$.  We define a sequence $(C_0, C_1, C_2, \ldots)$ in $\VRep(W)$
to be {\bf log concave} if, for all $i>0$, $C_i^{\otimes 2} - C_{i-1}\otimes C_{i+1} \in \Rep(W)$.
We call an element of $\grVRep(W)$ log concave if its sequence of coefficients is log concave.

\begin{remark}
If $W$ is the trivial group, then $\VRep(W)\cong \Z$ and this is the usual notion of log concavity for a sequence of integers.
\end{remark}

\begin{remark}
More generally, we can replace $\VRep(W)$ by any partially ordered ring and have a reasonable definition
of a log concave sequence in that ring.
\end{remark}

\begin{conjecture}\label{lc}
Let $W\curvearrowright M$ be an equivariant matroid.
\begin{enumerate}
\item The equivariant characteristic polynomial $H^W_M(t)$ is log concave.
\item The equivariant Kazhdan-Lusztig polynomial $P^W_M(t)$ is log concave.
\end{enumerate}
\end{conjecture}

\begin{remark}
When $W$ is the trivial group, Conjecture \ref{lc}(1) has existed in various forms since the 1960s,
and was only recently proven by Adiprasito, Huh, and Katz \cite{AHK}.  Conjecture \ref{lc}(2) for the trivial group
appeared in \cite[Conjecture 2.5]{EPW}.
\end{remark}

\begin{remark}
We have verified both parts of Conjecture \ref{lc} with a computer for the braid matroid $B_n$ for all $n\leq 9$,
as well as part 2 for the uniform matroid $U_{m,d}$ for all $m,d\leq 15$.  Part 1 of the conjecture
for uniform matroids is proved below (Proposition \ref{uniform-lc}) for arbitrary values of $m$ and $d$.
\end{remark}

\begin{remark}
In a forthcoming paper, we will explore the notion of log concavity for $W$-representions in greater detail.
There we will give many more conjectural examples of naturally arising log concave sequences of representations.
\end{remark}

\begin{proposition}\label{uniform-lc}
The equivariant characteristic polynomial $H_{m,d}(t)$ of the uniform matroid $U_{m,d}$ is log concave.
\end{proposition}

\begin{proof}
First, we observe that $OS^{S_{m+d}}_{U_{m,d},i}$ is equal to $OS^{S_{m+d}}_{U_{0,m+d},i}$ when $i<d$,
it is equal to a quotient of $OS^{S_{m+d}}_{U_{0,m+d},i}$ when $i=d$, and it is equal to zero when $i>d$.  For this reason,
it suffices to prove the proposition when $m=0$, in which case it says that 
$$\left(\wedge^i\C^d\right)^{\otimes 2} - \left(\wedge^{i-1}\C^d\otimes\wedge^{i+1}\C^d\right)\in\Rep(S_d).$$
Let $V_i := V[d-i,1^i]$.
Since $\wedge^i\C^d = V_i\oplus V_{i-1}$, it is sufficient to prove that
\begin{equation}\label{ii}
V_i^{\otimes 2} - V_{i-1}\otimes V_{i+1} \in \Rep(S_d)
\end{equation}
and
\begin{equation}\label{iimo}
V_{i-1}\otimes V_{i} - V_{i-2}\otimes V_{i+1} \in \Rep(S_d).
\end{equation}
These tensor products may be computed using a formula of Remmel \cite[Theorem 2.1]{Remmel}.
We will only prove Equation \eqref{ii}; the proof of Equation \eqref{iimo} is similar.

First, Remmel observes that tensoring with the sign representation takes $V_i$ to $V_{d-i-1}$,
and we may use this to reduce to the case where $i<d/2$, which we will assume for the remainder of the proof.
For any partition $\la$ of $d$, Remmel computes $$c(\la, i, j):= \dim \Hom(V[\la], V_i\otimes V_j);$$
we need to show that $c(\la, i, i) \geq c(\la, i-1,i+1)$ for all $\la$.

The number $c(\la, i, j)$ is zero unless $\la = [r,1^{d-r}]$ for some $r$ 
or $\la = [q,p,2^k,1^\ell]$ for some $q\geq p\geq 2$ and $k,\ell\geq 0$.
When $\la = [r,1^{d-r}]$, Remmel tells us that
$$c\left(\la, i, i\right) = \chi(d-2i-1\leq r \leq d)\and
c\left(\la, i-1, i+1\right) = \chi(d-2i-1\leq r \leq d-2),$$
where $\chi$ of a statement is 1 if the statement is true and 0 if the statement is false.
In particular, we can see that $c\left([r,1^{d-r}], i, i\right)\geq c\left([r,1^{d-r}], i-1, i+1\right)$.

When $\la = [q,p,2^k,1^\ell]$, we put
$$u = \max(p,d-2i),\qquad \omega = 2(d-i-k)-\ell,\qquad x =\lfloor\omega/2\rfloor,$$
$$v_1' = v_0 = \min(q,d-i-k-1),\qquad v_0' = \min(q,d-i-k-2),\and 
v_1 = \min(q,d-i-k).$$
Remmel tells us that
$$c\left(\la, i, i\right) =
\begin{cases}
0 \qquad\qquad\qquad\qquad\qquad\qquad\qquad\quad\;\text{if $p+k>d-i$}\\
\chi(u\leq x-1\leq v_0)+\chi(u\leq x\leq v_1) \;\;\text{if $p+k\leq d-i$ and $\ell$ is even}\\
\chi(u\leq x\leq v_0)+\chi(u\leq x\leq v_1) \;\qquad\text{if $p+k\leq d-i$ and $\ell$ is odd}
\end{cases}$$
and
$$c\left(\la, i-1, i+1\right) =
\begin{cases}
0 \qquad\qquad\qquad\qquad\qquad\qquad\qquad\quad\;\text{if $p+k>d-i-1$}\\
\chi(u\leq x-1\leq v_0')+\chi(u\leq x\leq v_1') \;\;\text{if $p+k\leq d-i-1$ and $\ell$ is even}\\
\chi(u\leq x\leq v_0')+\chi(u\leq x\leq v_1') \;\qquad\text{if $p+k\leq d-i-1$ and $\ell$ is odd.}
\end{cases}$$
Since $v_0'\leq v_0$, $v_1'\leq v_1$, and $d-i-1<d-i$, this implies that $c\left(\la, i, i\right)\geq c\left(\la, i-1, i+1\right)$.
\end{proof}
\excise{
\nicktodo{Notes:}
Note that our $V_i$ corresponds in Remmel's notation to $s = d-i$ (and $n=d$).
So we are interested in:  
\begin{itemize}
\item $s=d-i=t$ for the first term in Equation \eqref{ii}
\item $s=d-i-1$ and $t=d-i+1$ for the second term in Equation \eqref{ii}
\item $s=d-i$ and $t=d-i+1$ for the first term in Equation \eqref{iimo}
\item $s=d-i-1$ and $t=d-i+2$ for the second term of Equation \eqref{iimo}.
\end{itemize}
}

\begin{remark}
We define a sequence $(C_0, C_1, C_2, \ldots)$ in $\VRep(W)$
to be {\bf strongly log concave} if, for all $0\leq k\leq i\leq j$, $C_i\otimes C_j - C_{i-k}\otimes C_{j+k} \in \Rep(W)$.
We call an element of $\grVRep(W)$ strongly log concave if its sequence of coefficients is strongly log concave.

When $W$ is the trivial group and $C_i\geq 0$ for all $i$, strong log concavity is equivalent to log concavity with no internal zeros.
When $W=S_2$, however, the element $$f(t) := (1+2t+2t^2+t^3)V[2] + (3+2t+2t^2+3t^3)V[1,1]\in\grRep(W)$$ is log concave with
no internal zeros but not strongly log concave.

The notion of strong log concavity may be more natural than the notion of ordinary log concavity, 
since it has the property that strong log concavity of $f(t)$ and $g(t)$ implies
strong log concavity of $f(t)\otimes g(t)$.\footnote{The proof of this statement was communicated
to us by David Speyer, and will appear in a future paper.}  This fails for ordinary log concavity, as we can see by taking
$f(t)$ as above and $g(t) = (1+t)V[2]$.  
Conjecture \ref{lc} and Proposition \ref{uniform-lc} may both be generalized to the corresponding ``strong" versions.
\end{remark}

\bibliography{./symplectic}
\bibliographystyle{amsalpha}

\end{document}